\documentclass{article}
 
\usepackage[latin1]{inputenc}
\usepackage[T1]{fontenc}
\usepackage[francais, english]{babel}
\usepackage{lmodern}
\usepackage{stmaryrd}
\usepackage{amsthm}
\usepackage{amsmath}
\usepackage{amssymb}
\usepackage{mathrsfs}
\usepackage{soul}
\usepackage{layout}
\usepackage{setspace}
\usepackage[top=3cm, bottom=3cm, left=3.3 cm, right=3.3 cm]{geometry}
\usepackage{pgf,tikz}
\usepackage{graphicx}
\usepackage{xcolor}
\usetikzlibrary{arrows}
\usepackage{listings}
\usepackage{enumitem}
\usepackage{hyperref}
\usepackage{cases}
\usepackage{mathabx}
\usepackage{cleveref}
\usepackage{xcolor}
\usepackage{appendix}
\usepackage{fancyhdr}
\usepackage{sectsty}
\usepackage{comment}

\usepackage{caption}
\usepackage{graphicx, threeparttable}
 \DeclareCaptionFormat{sanslabel}{#3}%

\newtheorem{th.}{Theorem}[section]
\newtheorem{lemme}[th.]{Lemma}
 
\newtheorem*{lemme*}{Lemma} 
\newtheorem*{pgd*}{Principe des grandes déviations} 
\newtheorem{prop}[th.]{Proposition}
\newtheorem*{prop*}{Proposition}
\newtheorem*{prop5.3bis*}{Proposition 5.3 bis}
\newtheorem*{def*}{Definition}
\newtheorem*{cor*}{Corollary}
\newtheorem{cor.}[th.]{Corollary}

\newtheorem*{rem*}{Remark}
\newtheorem*{rem.}{Remark}
\newtheorem*{ex.}{Example}
\newtheorem*{th*}{Theorem}

\newtheorem*{thderiv-hypsurf}{\Cref{drift-hypsurf}}
\newtheorem*{threc-hypsurf}{\Cref{rec-hypsurf}}
\newtheorem*{th.*}{Théorème}
\newtheorem*{thderiv1}{ \Cref{TH-deriv}}
\newtheorem*{thderivnonpath}{ \Cref{drift-non path}}

\newcommand{\C}{\mathbb{C}}
\newcommand{\R}{\mathbb{R}}
\newcommand{\Z}{\mathbb{Z}}

\newcommand{\N}{\mathbb{N}}

\newcommand{\ad}{\text{ad}}

\newcommand{\supp}{\text{supp }}

\subsectionfont{\fontsize{13}{17}\selectfont}

\subsubsectionfont{\fontsize{12}{15}\selectfont}

\makeatletter
\newcounter {subsubsubsection}[subsubsection]
\renewcommand\thesubsubsubsection{\thesubsubsection .\@alph\c@subsubsubsection}
\newcommand\subsubsubsection{\@startsection{subsubsubsection}{4}{\z@}%
                                     {-3.25ex\@plus -1ex \@minus -.2ex}%
                                     {1.5ex \@plus .2ex}%
                                     {\normalfont\large\bfseries}}
\newcommand*\l@subsubsubsection{\@dottedtocline{3}{10.0em}{4.1em}}
\newcommand*{\subsubsubsectionmark}[1]{}
\makeatother

\title{Drift of random walks on abelian covers of finite volume homogeneous spaces}
\author{Timothée Bénard}
\date{}

\begin{document}
\large
\maketitle






\pagestyle{plain}

\bigskip

\begin{abstract} 
Let $G$ be a connected simple real Lie group, $\Lambda_{0}\subseteq G$ a lattice without torsion and $\Lambda \unlhd \Lambda_{0}$ a normal subgroup such that $\Lambda_{0}/\Lambda\simeq \Z^d$. We study the drift of a random walk on the $\Z^d$-cover $\Lambda\backslash G$ of  the finite volume homogeneous space $\Lambda_{0}\backslash G$. This walk is defined by a Zariski-dense compactly supported probability measure $\mu$ on $G$.  We first assume the covering map $\Lambda\backslash G\rightarrow \Lambda_{0}\backslash G$ does not unfold any cusp of  $\Lambda_{0}\backslash G$ and compute the drift at \emph{every} starting point. Then we remove this assumption and describe the drift almost everywhere. The case of hyperbolic manifolds of dimension 2 stands out with non-converging type behaviors. The recurrence of the trajectories is also characterized in this context.
\end{abstract} 

\bigskip

\baselineskip=15pt

\large
\section*{Introduction}
\bigskip

$\Z^d$-covers of finite volume hyperbolic surfaces are important examples to study dynamical systems in infinite measure. They have the advantage of being very concrete,  allow to use  tools developed for  finite volume spaces via the periodicity of the covering, but still present a true complexity leading sometimes to unexpected developments. 

Such covers have been extensively studied in terms of their geodesic flow or horocycle flow.  Let us mention the work of Sinai \cite{Sinai60}, followed by Le Jan, Guivarc'h, Enriquez \cite{LeJan92, GuivarchLeJan93, LeJan94,  Enr-LeJan} revealing that the distribution of a fundamental domain spreading under the action of the geodesic flow obeys a central limit theorem (partially Gaussian, partially Cauchy). On a similar subject, Oh and Pan  obtained a local  limit theorem in  \cite{PanOh, Panlocalmixing}.  Concerning the horocycle flow, spectacular progress have been achieved by Sarig in \cite{Sar}, adding the final piece to the classification of  Radon invariant measures initiated by Babillot-Ledrappier \cite{BabLed, Bab-classif}, and realizing a first step toward an extension of Ratner theory for infinite volume homogeneous spaces.

 The present paper adopts a different point of view, looking into the dynamics of  random walks on  $\Z^d$-covers of  finite volume homogeneous spaces. Even if the dynamics  is no longer deterministic, it shares strong relations with the works cited above, which will be key ingredients in several proofs below.  Recent investigations  have  been carried out in this context by Conze and Guivarch', who described in \cite{ConGui}    the recurrence properties of symmetric random walks when the base of the cover is compact.  Our purpose is to examine the rate of escape of a walk trajectory at linear scale, in the spirit of a law of large numbers. Loosely speaking we will show that the drift of a (not necessarily symmetric) walk is always null,  except when  the base of the cover is a hyperbolic surface with unfolded cusps, in which case  non converging  behaviors occur.

\bigskip

The homogeneous spaces we consider are given by quotients of quasi-simple Lie groups. We denote by $G$  a connected real Lie group with simple Lie algebra and finite center, $\Lambda_{0}\subseteq G$ a  lattice without torsion,  $\Lambda\unlhd \Lambda_{0}$ a normal subgroup such that  $\Lambda_{0}/\Lambda\simeq \Z^d$, and look into the  quotient space $X=\Lambda \backslash G$, which is a $\Z^d$-cover of the finite volume homogeneous space $X_{0}=\Lambda_{0}\backslash G$.   This context is actually very explicit as we can assume without loss of generality that $G=SO_{e}(1,m)$ or $G=SU(1,m)$ for some integer $m\geq 2$ (see \ref{sec1.1}). 
In this case,  $X_{0}$ corresponds to the orthonormal frame bundle of a real or complex hyperbolic manifold  of finite volume $M_{0}$, and $X$ to the orthonormal  frame bundle of some $\Z^d$-cover $M$ of $M_{0}$.

To define a random walk on $X$ or $X_{0}$, we choose a probability measure $\mu$ on $G$ and define the transitional probability measure at a point $x$ as the convolution 
$\delta_{x}\star \mu$, i.e. as the image of $\mu$ by the map $g\mapsto xg$. Set $B=G^{\N^\star}$, $\beta=\mu^{\otimes \N^\star}$. A typical trajectory for the $\mu$-walk starting from a point $x\in X$ is thus obtained  by choosing an element $b\in B$ with law $\beta$ and considering the sequence $(xb_{1}\dots b_{n})_{n\geq 0}\in X^\N$. In the following, we will always assume that the support of the measure $\mu$ is \emph{compact} and generates a \emph{Zariski-dense} sub-semigroup $\Gamma_{\mu}$ in $G$. 


To express the rate of escape of the $\mu$-walk on $X$, we fix a fundamental domain $D$ for the action of $\Z^d$ on $X$ such that $D$ lifts well the cusps of $X_{0}$ (see \ref{fund}), and some $\Z^d$-equivariant  measurable map $i : X \rightarrow \R^d$ that is bounded on $D$. The role of this map is to associate  to every point $x\in X$ some coordinates $i(x)$ in $\R^d$ that quantify which $\Z^d$-translate of $D$ contains $x$.

The law of large numbers now inspires the following questions. Let $(xb_{1}\dots b_{n})_{n\geq 0}$ be a typical trajectory of the $\mu$-walk on $X$. 

\bigskip

\emph{Does the sequence  $\frac{1}{n}i(xb_{1}\dots b_{n})_{n\geq 1}$ converge in $\R^d$? And if so, what is the limit? }

\bigskip

To discuss this issue, we first clarify the notion of \emph{drift}.  Let $x\in X$ be a point on $X$ and $E$ a subset of $\R^d$.  We say the $\mu$-walk with starting point $x$ has drift $E$ if for $\beta$-almost every $b\in B$,  the set of accumulation points of the normalized  sequence of positions $\left(\frac{1}{n}i(xb_{1}\dots b_{n})\right)_{n\geq 1}$ is equal to $E$. When $E$ is a singleton, we recover the more common definition of drift, as it appears in the classical law of large numbers. However, we will see examples where $E$ is not reduced to a point.  

The  drift of the $\mu$-walk on $X$ is linked with the dynamical properties of a cocycle on $X_{0}\times G$ that we now introduce. Notice first that  each step of the walk yields a variation of the index of position given by the cocycle 
$$ X\times G\rightarrow \R^d, \,\,(x,g)\mapsto i(xg)-i(x)$$
The assumption that $i$ commutes with the action of $\Z^d$ implies this cocycle is $\Z^d$-invariant, hence defines a quotient \emph{drift cocycle} :
 $$\sigma : X_{0}\times G\rightarrow \R^d, \,\,(x+\Z^d,g)\mapsto i(xg)-i(x)$$
As for $x\in X$, $b\in B$, one has $i(xb_{1}\dots b_{n})=\sigma( x+\Z^d, b_{1}\dots b_{n})+i(x)$,  the drift at $x$ is also the set of accumulation points of  $\beta$-typical sequences $\left(\frac{1}{n}\sigma(x+\Z^d, b_{1}\dots b_{n})\right)_{n \geq 1}$. Hence, we will freely talk about  the drift of the $\mu$-walk at a point that is not on $X$, but on $X_{0}$. 

Now, the ergodicity\footnote{It is a consequence of  \cite[Proposition 2.8]{BQRW}, the strict convexity of balls in a Hilbert space, and Howe-Moore Theorem \cite[Theorem 2.2.15]{Zim84}.}of the $\mu$-walk on $X_{0}$ with respect to the Haar probability measure $\lambda_{0}$ gives a first answer to our questions :  \emph{if $\sigma$ is $\lambda_{0}\otimes \mu$-integrable}, then by Birkhoff Ergodic Theorem, for $\lambda_{0}$-almost every $x_{0}\in X_{0}$, $\beta$-almost every $b\in B$, 
$$ \frac{1}{n}\sigma(x_{0}, b_{1}\dots b_{n}) \underset{n\to +\infty}{\longrightarrow} \int_{X_{0}\times G}\sigma(x,g) \,d\lambda_{0}(x) d\mu(g) $$
In other words, the drift of the $\mu$-walk on $X$ is almost everywhere well defined and equal to the expectation of the drift cocycle $\sigma$ for the probability  measure $\lambda_{0}\otimes \mu$.

\bigskip

\Cref{sec2} strengthens this conclusion under the more precise assumption that the cusps of the base $X_{0}$ are not unfolded in the cover $X$ (see Figure 1). In this case,  the drift is null at every point in $X_{0}$ whose $\Gamma_{\mu}$-orbit is infinite, and can be computed explicitly otherwise. In particular it is well-defined, explicit  for \emph{every starting point}, and null outside of a countable subset of $X_{0}$. This result draws a sharp contrast with the usual law of large numbers on $\R^d$ where the drift depends strongly on the law of increment.  The proof relies on the classification of  stationary probability measures given by Benoist-Quint   \cite{BQII}. 

\bigskip

 \begin{thderiv1}
 

Assume no cusp of $X_{0}$ is unfolded in $X$. Given $x_{0} \in X_{0}$, we may distinguish two cases :

 \smallskip
 
 \begin{enumerate}
 \item If the orbit $x_{0}\Gamma_{\mu}$ is infinite in $X_{0}$ then the drift is null : for $\beta$-almost every $b \in B$, one has 
 $$\frac{1}{n}\sigma(x_{0}, b_{1}\dots b_{n})\underset{n\to +\infty}{\longrightarrow} 0$$

 \item If the orbit $x_{0}\Gamma_{\mu}$ is finite in $X_{0}$, then for $\beta$-almost every $b \in B$, one has $$\frac{1}{n}\sigma(x_{0}, b_{1}\dots b_{n})\underset{n\to +\infty}{\longrightarrow} \mu\otimes\nu_{0}(\sigma)$$
 
 where $\nu_{0}$ is  the uniform probability measure on $x_{0}\Gamma_{\mu}$ given by 
$\nu_{0} := \frac{1}{\sharp x_{0}\Gamma_{\mu} }\sum_{y \in x_{0}\Gamma_{\mu}}\delta_{ y}$.
 \end{enumerate}

 \end{thderiv1}

 \begin{figure}[!h]\centering \captionsetup{format=sanslabel}
\begin{measuredfigure}
\includegraphics[scale = 1.5]{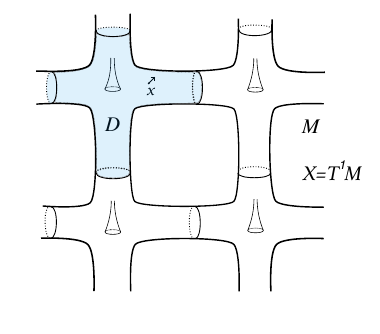}
\caption{Figure 1 : Example where $G=SO_{e}(1,2)$, $M_{0}$ is a hyperbolic surface of genus 2 with one cusp, and $M$ is a $\mathbb{Z}^2$-cover of $M_{0}$ that does not unfold the cusp. }
\end{measuredfigure}
\end{figure}

\bigskip

\bigskip

\Cref{sec3} is independent from \Cref{sec2} and characterizes the drift at almost every point of $X$, without making  assumptions on the cusps of $X_{0}$.  We see in \Cref{Case1} that $\sigma$ is still $\lambda_{0}\otimes\mu $-integrable if $G\neq SO_{e}(1,2)$ and infer with a short argument that the drift of a typical trajectory is null.

\bigskip
\begin{thderivnonpath}
 If $G\neq SO_{e}(1,2)$, then for almost every $x_{0}\in X_{0}$, $\beta$-almost every $b\in B$, 
$$\frac{1}{n}\sigma(x_{0}, b_{1}\dots b_{n})\underset{n\to +\infty}{\longrightarrow}0 $$

\end{thderivnonpath}

\bigskip

However, a radically different behavior can be observed if $G=SO_{e}(1,2)$. In this case, $X$ corresponds to the unit tangent bundle of a $\Z^d$-cover $S$ of a finite volume hyperbolic surface $S_{0}$.  We show that for almost every $x\in X\equiv T^1S$, $\beta$-almost every $b\in B$, the sequence $\left(\frac{1}{n}\sigma(x, b_{1}\dots b_{n})\right)_{n\geq 1}$ needs not converge in $\R^d$, but accumulates over a subspace of $\R^d$ generated by the direction of translations above the unfolded cusps of $S_{0}$. 
More precisely, denote by $\mathscr{C}_{1}, \dots, \mathscr{C}_{s}$ the cusps of $S_{0}$, and by $v_{1}, \dots,v_{s}\in \Z^d$ the translations obtained by lifting to $S$ simple closed curves in $S_{0}$ homotopic to  the  cusps.  Set $E_{C}=\text{Vect}_{\R}\{v_{1},\dots v_{s}\}$ the sub-vector space of $\R^d$ spanned by the $v_{i}$'s.

\bigskip

\begin{thderiv-hypsurf}
The drift of the $\mu$-walk on $T^1S$ is almost everywhere equal to  $E_{C}$ :  for almost every $x_{0}\in T^1S_{0}$, for $\beta$-almost every $b\in B$,
$$\left\{\text{accumulation points of the sequence $\left(\frac{1}{n}\sigma(x_{0}, b_{1}\dots b_{n})\right) _{n\geq 1}$}\right\} = E_{C} $$
\end{thderiv-hypsurf}

\bigskip
For example, if  $S$ is a $\Z$-cover of a hyperbolic three-holed sphere, then the drift of the walk is almost everywhere  equal to $\R$.

\begin{figure}[h]\centering \captionsetup{format=sanslabel}
\begin{measuredfigure}
\includegraphics[scale=0.8]{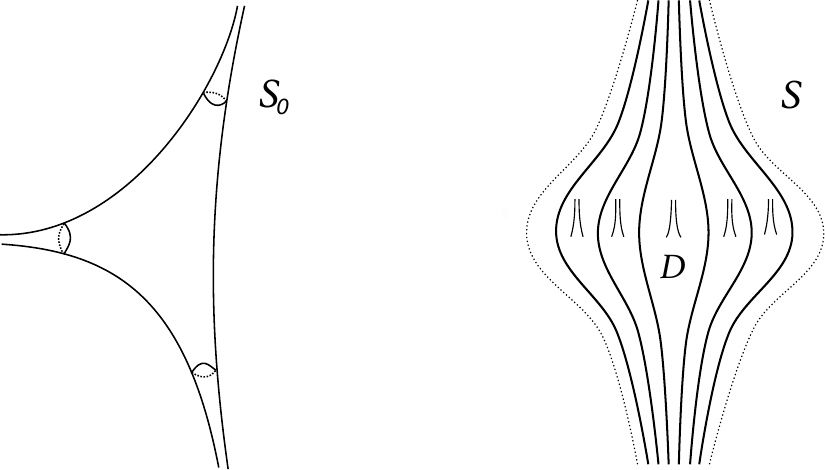}
\caption{Figure 2 : A hyperbolic three-holed  sphere $S_{0}$  and a $\mathbb{Z}$-cover $S$}
\end{measuredfigure}
\end{figure}

\bigskip
The key to \Cref{drift-hypsurf} is a result by Enriquez and Le Jan that describes the statistic of the winding of the geodesic flow around the different cusps of $S_{0}$ \cite{Enr-LeJan}. We infer a \emph{central limit theorem} stating that the random variable $\frac{1}{n}\sigma(x_{0},b_{1}\dots b_{n})$, where $x$ varies with law $\lambda_{0}$ and $b$ varies with law $\beta$, behaves for large $n$ as a centered Cauchy distribution on $E_{C}$ (\Cref{Cauchy}).


Our proof of \Cref{drift-hypsurf}  also yields a characterization of $\Z^d$-covers of finite volume hyperbolic surfaces for which the $\mu$-walk is recurrent.   

\begin{threc-hypsurf}
The $\mu$-walk on $T^1S$ is recurrent and ergodic if and only if $d=1$, or $d=2$ and $\dim E_{C}=0$; and it is transient otherwise.
\end{threc-hypsurf}

This last result extends  Conze-Guivarc'h recurrence criterion mentioned above which dealt with the case where $\mu$ is symmetric and $S_{0}$ is compact \cite[Proposition 4.5]{ConGui}.  We recently generalized  \Cref{rec-hypsurf} in \cite{Tim-recurrence} to the context of random walks on rank one symmetric spaces.   The proof is however much more involved than in the present case, and uses a different technique, approximating  $\mu$-trajectories by their limit geodesic ray.

\bigskip

\emph{Structure of the paper.}

\Cref{sec1} depicts the geometry of the homogeneous spaces  we consider, defining notably the notion of cusp, and of good fundamental domain.

\Cref{sec2} deals with  $\Z^d$-covers $X$ which do not unfold the cusps of their base $X_{0}$, and proves \Cref{TH-deriv} characterizing the drift of a walk at every starting point. 

\Cref{sec3} makes no assumption on the cusps of the base $X_{0}$ and describes the drift of a walk at almost every starting point. We start with the proof of \Cref{drift-non path} then focus on the case where the base $X_{0}$ is the unit bundle of a hyperbolic surface with unfolded cusps. We  prove in \Cref{Cauchy} a central limit theorem  for the drift cocycle $\sigma$, and deduce Theorems \ref{drift-hypsurf} and \ref{rec-hypsurf}.

\bigskip

\tableofcontents

\newpage
\section{Geometry of the ambient space} \label{sec1}

We first explain in greater details the geometric notions involved in this paper.

\subsection{Real or complex hyperbolic manifolds}  \label{sec1.1}
As above, let $G$ be a connected Lie group with simple Lie algebra and finite center, and $\Lambda \unlhd \Lambda_{0} \subseteq G$ torsion-free discrete subgroups of $G$, such that $\Lambda_{0}$ has finite covolume in $G$  and  $\Lambda_{0}/\Lambda \simeq\Z^d$ (where $d\in \N^\star$).
The existence of such subgroups $\Lambda$ and $\Lambda_{0}$ forces the Lie algebra $\mathfrak{g}$ of $G$ to be isomorphic to $\mathfrak{so}(1, m)$ or $\mathfrak{su}(1,m)$ for some $m\geq 2$ [Indeed, if it is not the case, then the group $G$ has Kazhdan's property (T) (\cite{BekHarVal}, page 8), as does $\Lambda_{0}$ which is a lattice in $G$, and therefore its quotient $\Lambda_{0}/\Lambda \simeq\Z^d$. Absurd]. Hence, we may assume without loss of generality that 
$$G=SO_{e}(1,m) \,\,\, \,\,\, \text{or} \,\,\, \,\,\,G=SU(1,m) $$

Denote by $K\subseteq G$ a maximal compact subgroup of $G$,    e.g.  $K:=\{1\}\times SO(m)$ in the first case  or  $K:=S(U(1)\times U(n))\simeq U(n)$ in the second case. Let $\mathfrak{g}, \mathfrak{k}$ the respective Lie algebras of $G$ and $K$, set $\mathfrak{s}=\mathfrak{k}^{\perp}$ the orthogonal of $\mathfrak{k}$ in $\mathfrak{g}$ for the Killing form $B : \mathfrak{g}\times \mathfrak{g}\rightarrow \R$. We then have a decomposition $\mathfrak{g}=\mathfrak{k}\oplus \mathfrak{s}$ and $B$ is negative (resp. positive) definite on $\mathfrak{k}$ (resp. $\mathfrak{s}$). Now endow the quotient space $G/K$ with the unique riemannian metric that is left $G$-invariant and coincides with $B_{|\mathfrak{s}\times \mathfrak{s}}$ on $T_{e}G/K\equiv \mathfrak{s}$. The resulting riemannian manifold is the \emph{symmetric space} associated to $G$ (see \cite{Hel-sym}).

In the case where $G=SO_{e}(1,m)$, this construction leads to the real hyperbolic space of dimension $m$. The double-quotient $M_{0}=\Lambda_{0}\backslash G/K$ is a finite volume hyperbolic manifold, $M=\Lambda \backslash G/K$ is a $\Z^d$-cover of $M_{0}$ and the spaces $X_{0}:=\Lambda_{0} \backslash G$ and  $X:= \Lambda  \backslash G $ correspond to the orthonormal frame bundles of  $M_{0}$ and $M$. If $G=SU(1,m)$, we can draw the same conclusions except that one needs to consider this time complex hyperbolic manifolds.

\subsection{The notion of cusp} \label{sec1.2}

We first recall the Iwasawa decomposition \cite[5.1]{Hel-sym}. Let $v_{0}\in \mathfrak{s}$ be some non-zero element of $\mathfrak{s}$ and write $\mathfrak{a}=\R v_{0}$ the associated Cartan subspace. The adjoint action $\ad v_{0} : \mathfrak{g}\rightarrow \mathfrak{g}$ is diagonalizable. Set $\mathfrak{n}\subseteq \mathfrak{g}$ the sum of its eigenspaces with strictly positive eigenvalue. All the elements in $\mathfrak{n}$ are nilpotent matrices. We set $A:=\exp(\mathfrak{a})$ and  $N:=\exp(\mathfrak{n})$ the unique connected Lie subgroups of $G$ of  Lie algebras $\mathfrak{a}$ and $\mathfrak{n}$. For $t\in \R$, write also $a_{t}:=\exp(tv_{0})$.

\begin{th*}[Iwasawa decomposition]
The map 
$$N\times \R \times K \rightarrow G, \,\,(n,t,k)\mapsto na_{t}k$$
is a diffeomorphism.
\end{th*}

Let us now recall the notion of cusp for the finite volume manifold $M_{0}$. One can partition $M_{0}$ as $$M_{0}= L_{0}\sqcup  \bigsqcup_{j=1,\dots, q} \mathscr{C}_{j}$$
where $L_{0}$ is a compact submanifold with boundary, and each  $\mathscr{C}_{j}$ is an open subset whose structure is as follows. Denote by $\pi_{0}: G \rightarrow M_{0}$ the projection map, and  set $A^{++}=\{a_{t}, t>0\}$. Then there exists an element $g_{j}\in G$ such that the restriction 
$$\pi_{0} : g_{j}NA^{++} \rightarrow \mathscr{C}_{j}$$ 
is well defined, surjective. Moreover, if one sets $N^{\Lambda_{0}}_{j}=\{g\in \Lambda_{0}, gg_{j}N=g_{j}N\}$, then the action of $N^{\Lambda_{0}}_{j}$ on $g_{j}N$ is properly discontinuous, free,  cocompact, and $\pi_{0}$ defines a quotient diffeomorphism : 
$$\pi_{0} :  N^{\Lambda_{0}}_{j} \backslash g_{j}N A^{++} \rightarrow \mathscr{C}_{j}$$

\bigskip

A cusp $\mathscr{C}_{j}$ in $M_{0}$ is said to be \emph{unfolded} in $M$ if there exists  a  closed path on $\mathscr{C}_{j}$ whose lift to $M$ is not closed.  We can reformulate this notion algebraically. Set $N^{\Lambda}_{j}=\{g\in \Lambda, gg_{j}N=g_{j}N\}$. Then $N^{\Lambda}_{j}$  is a normal subgroup of $N^{\Lambda_{0}}_{j}$ and $N^{\Lambda_{0}}_{j}/N^{\Lambda}_{j} \hookrightarrow \Z^d$. The cusp $\mathscr{C}_{j}$ is unfolded in $M$ if and only if one has a strict inclusion $N^{\Lambda}_{j} \subsetneq N^{\Lambda_{0}}_{j}$

\bigskip

\subsection{Fundamental domain} \label{fund}

The riemannian manifold $M$ admits a free and proper action of $ \Z^d$ by isometries such that $\Z^d \backslash M\simeq M_{0}$. We  denote by $D\subseteq M$ a fundamental domain for this action. This means that  $D$ is a closed  subset of $M$ such that :

\begin{itemize} 
\item[1)] the  translates $(D+k)_{k\in \Z^d}$ form a locally finite covering of $M$
\item[2)] their interiors $(\overset{\circ}{D}+k)_{k\in \Z^d}$ are mutually disjoint and have their  union of full measure in $M$
\end{itemize}

Sometimes, we will also denote by $D$ the corresponding $K$-invariant subset of $X$. 

\bigskip
Assumptions 1) and 2) are classical but not sufficient to avoid  pathological behaviors which are irrelevant for our description of the drift. Hence, we will also ask that for every point  $p\in D$ whose projection on $M_{0}$ belongs to  a cusp boundary $\partial\mathscr{\mathscr{C}}_{j}$, the associated geodesic ray going to infinity in $\mathscr{C}_{j}$ is also contained in $D$ (as in Figure 1 or 2). More precisely, we say that a fundamental domain $D$ is \emph{good} if :

\begin{itemize} 
\item[3)] for every $j\in\{1,\dots, q\}$, there exists a fundamental domain $F^0_{j}\subseteq g_{j}N$ for the action of $N^{\Lambda_{0}}_{j}$ on $g_{j}N$ such that the domain $F^0_{j}A^{++}$ projects onto the preimage of $\mathscr{C}_{j}$ in $D$. 
\end{itemize}

A first pleasant feature of this additional assumption is the following observation. 
If $D$ is a good fundamental domain,  $i: M\rightarrow \R^d$ is a measurable $\Z^d$-equivariant map that is bounded on $D$, and $(p_{n})_{n \geq 0 }\in M^\N$ is a sequence of points in $M$, then \emph{the asymptotic behavior of $\frac{1}{n}i(p_{n})$ does not depend on the choices of $D$ or $i$. In other words the notion of drift is intrinsic}. Indeed, if $(D',i')$ is a couple of other candidates for $D$ and $i$, then $D$ is covered by a finite number of translates $(D'+k)_{k\in \Z^d}$ (consequence of 3 and the compactness of $L_{0}$). Thus $i'$ is bounded on $D'$, and by $\Z^d$-equivariance of $i$ and $i'$, the difference $i-i'$ is globally bounded, which proves the claim.



\section{Drift at every point} \label{sec2}

Let $M=\Lambda \backslash G/K$ be a $\Z^d$-cover of a finite volume - real or complex  - hyperbolic manifold $M_{0}=\Lambda_{0} \backslash G/K$ as in \Cref{sec1.1}, and $X=\Lambda\backslash G$, $X_{0}=\Lambda_{0}\backslash G$ their orthonormal frame bundles. We assume that the cusps of $M_{0}$ are not unfolded in $M$ and characterize the drift of a random walk on $X$ for \emph{every} starting point.

\subsection{Statement and consequences}

We first recall and comment on  \Cref{TH-deriv}.

\bigskip

Fix a probability measure $\mu$  on $G$ whose support is compact and generates a Zariski-dense semigroup $\Gamma_{\mu}$ in $G$. Set $B=G^{\N^\star}$, $\beta=\mu^{\otimes \N^\star}$. Let $D$ be a good fundamental domain for the $\Z^d$-action on $M$ and denote by $i: M\rightarrow \R^d$  a measurable $\Z^d$-equivariant map that is bounded on $D$. We also see $i$ as a $\Z^d$-equivariant $K$-invariant  map on $X$ and write $\sigma : X_{0}\times G \rightarrow \R^d$ the associated drift cocycle (cf.  introduction). The goal of the section is to show the following.

 \begin{th.}\label{TH-deriv}

Assume that no cusp of $M_{0}$ is unfolded in the cover $M$. Given $x_{0} \in X_{0}$, we may distinguish two cases :
 \smallskip
 
 \begin{enumerate}
 \item If the orbit $x_{0}. \Gamma_{\mu}$ is infinite in $X_{0}$ then the drift is null : for $\beta$-almost every $b \in B$, one has 
 $$\frac{1}{n}\sigma(x_{0},b_{1}\dots b_{n})\underset{n\to +\infty}{\longrightarrow} 0$$

 \item If the orbit $x_{0}.\Gamma_{\mu}$ is finite in $X_{0}$, then for $\beta$-almost every $b \in B$, one has $$\frac{1}{n}\sigma(x_{0},b_{1}\dots b_{n})\underset{n\to +\infty}{\longrightarrow} \nu_{0}\otimes \mu(\sigma)$$
 
 where $\nu_{0}$ is  the uniform probability measure on $x_{0}. \Gamma_{\mu}$ given by 
$\nu_{0} := \frac{1}{\sharp \Gamma_{\mu}.x_{0} }\sum_{y \in \Gamma_{\mu}.x_{0}}\delta_{ y}$.
 \end{enumerate}
 
 \end{th.}

As a consequence of \Cref{TH-deriv}, we obtain that the drift is null at every point  if the probability measure $\mu$ is symmetric, i.e. if $\mu$ is invariant under the inversion map $G\rightarrow G, g\mapsto g^{-1}$.

\bigskip

\begin{cor.}
If no cusp of $M_{0}$ is unfolded in the cover $M$, and if the probability measure $\mu$ is symmetric, then for every $x_{0}\in X_{0}$, $\beta$-almost every $b \in B$, one has $$\frac{1}{n}\sigma(x_{0}, b_{1}\dots b_{n})\underset{n\to +\infty}{\longrightarrow} 0$$

\end{cor.}

\bigskip
\begin{proof}[Proof of the corollary]
According to \Cref{TH-deriv}, we just need to check that if $x_{0}$ has  finite $\Gamma_{\mu}$-orbit, then $ \nu_{0}\otimes\mu(\sigma)=0$. Set $\omega :=x_{0}.\Gamma_{\mu} \subseteq X_{0}$. Using the symmetry of $\mu$ and the  $\Gamma_{\mu}$-invariance of $\nu_{0}$, one computes that  
\begin{align*}
\nu_{0} \otimes \mu (\sigma) &= \frac{1}{2} \int_{X_{0} \times G} \sigma(y_{0},g) +\sigma(y_{0}, g^{-1}) d\mu(g) d\nu_{0}(y_{0})\\
         &= \frac{1}{2} \int_{G \times X_{0}} \sigma(y_{0},g) +\sigma(y_{0}g, g^{-1}) d\mu(g) d\nu_{0}(y_{0})\\
        &= 0
\end{align*}
as  $\sigma(y_{0},g) +\sigma(y_{0}g, g^{-1})=0$ for all $y_{0} \in X_{0}$, $g\in G$, according to the cocycle relation.

\end{proof}

\bigskip

\bigskip

\noindent{\bf Example with a non zero drift.}
We construct such an example in the case where $G= SO_{e}(1,2) \equiv PSL_{2}(\R)$, i.e. by considering an abelian cover of a hyperbolic  surface. Let $S_{0}$ be a compact hyperbolic surface of genus $2$. Denote by $c_{1},c_{2}$ the two simple closed (non-parametrized) geodesic curves on $S_{0}$ pictured in the figure below, and $p_{0}\in S_{0}$ their unique point of intersection. Up to choosing a good surface $S_{0}$, one can assume these curves intersect orthogonally at $p_{0}$. 

\bigskip
\begin{figure}[h]\centering \captionsetup{format=sanslabel}
\begin{measuredfigure}
\includegraphics[scale=0.9]{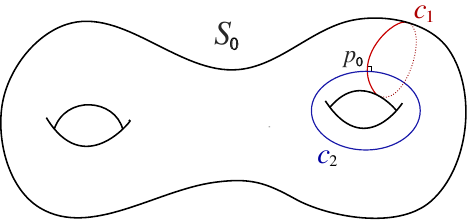}
\caption{Figure 3 : The hyperbolic surface $S_{0}$ and the geodesic curves $c_{1},c_{2}$ intersecting at $p_{0}$ }
\end{measuredfigure}
\end{figure}

\bigskip
Now construct a $\Z$-cover of $S_{0}$ in the following way. Cut $S_{0}$ along the curve $c_{1}$ to get a hyperbolic surface $\Sigma_{0}$ of genus $1$, with two boundary components  $\delta_{1}$, $\delta_{2}$. Let $(\Sigma_{i})_{i\in \Z}$ be copies of this surface, and denote by $\delta_{i,1}$, $\delta_{i,2}$ the corresponding boundary components. Now glue together the surfaces  $\Sigma_{i}$, by identifying each boundary component $\delta_{i,2}$ with the $\delta_{i+1,1}$ (and in the same way as $\delta_{2}$ identifies to $\delta_{1}$). The surface obtained is denoted by $S$, and can be seen naturally as a  $\Z$-cover of $S_{0}$. Notice also that the geodesic segments corresponding to $c_{2}$ in each $\Sigma_{i}$ are glued together in the process into a global geodesic arc on $S$ that connects the two ends of the cover.

We now define a suitable  random walk on $T^1S$. Given $t, \theta \in \R$, set 
$$a_{t}:=\begin{bmatrix}e^{t/2}& 0 \\ 0&e^{-t/2} \end{bmatrix}, \,\,\,\,\,\,\,\,\,\,\,\,
R_{\theta} :=\begin{bmatrix}\cos(\theta/2)& -\sin(\theta/2) \\ \sin(\theta/2)&\cos(\theta/2) \end{bmatrix}$$

If we denote by $X_{0}$ the unit tangent bundle of $S_{0}$, and see it as a homogeneous space  $X_{0}=\Lambda_{0} \backslash G$  where $\Lambda_{0}\subseteq G$ is a torsion-free discrete subgroup, then the (right-) action of an element $a_{t}$ on a vector $x_{0}\in X_{0}$ moves  $x_{0}$
as the geodesic flow at time $t$, whereas $R_{\theta}$ rotates $x_{0}$ by an angle $\theta$ without moving its base point. Denote by $l_{1}, l_{2}\in \R_{>0}$ the length of the geodesic curves $c_{1}, c_{2}$. We define a probability measure on $G$ by setting : 
$$\mu:= \frac{1}{2}(\delta_{g_{1}}+\delta_{g_{2}}) \,\,\,\,\,\,\text{ where $g_{1}:=R_{-\pi/2}a_{l_{1}}R_{\pi/2}$, and $g_{2}:=a_{l_{2}}$}$$.

As the matrices $g_{1}, g_{2}$ are hyperbolic and do not commute with each other, they must generate a Zariski-dense sub-semigroup $\Gamma_{\mu}$  in $G$. Consider a vector  $x_{0} \in X_{0}$  with base point $p_{0}$ and direction tangent to $c_{2}$ (hence orthogonal to $c_{1}$). The $\Gamma_{\mu}$-orbit of $x_{0}$ is reduced to a point : $x_{0}.\Gamma_{\mu}=x_{0}$. Moreover, for any map $i$  defining the drift cocycle $\sigma$, we have by $\Z$-equivariance $\sigma(x_{0}, g_{1})=0$ and $\sigma(x_{0}, g_{2})= 1$ (up to choosing rather $-x_{0}$ from the start). Hence, for $\beta$-almost every $b \in B$, 
$$\sigma(x_{0}, b_{1}\dots b_{n})= \sharp\{j\leq n, \,b_{j}=g_{2}\} \simeq \frac{n}{2}$$

The drift is $1/2$.

\bigskip
\subsection{Proof of \Cref{TH-deriv}}

Let us prove \Cref{TH-deriv}.  We recall the notations $B:=G^{\N^\star}$, $\beta:=\mu^{\otimes \N^\star}$, and set $T:B\rightarrow B, (b_{i})_{i\geq 1}\mapsto (b_{i+1})_{i\geq 1}$ the one-sided shift.

\bigskip

\bigskip

We first deal with the \underline{second case} as it is elementary. Let $x_{0}\in X_{0}$ be a point whose $\Gamma_{\mu}$-orbit $\omega:=x_{0}. \Gamma_{\mu}$ is finite. The  uniform probability measure  $\nu_{0}$ on $\omega$ is then $\mu$-stationary and ergodic. Hence the  dynamical system $(B^{X_{0}}, \beta^{X_{0}}, T^{X_{0}})$ defined by 
\begin{align*}
B^{X_{0}}:= X_{0} \times B, \,\, \,\,\beta^{X_{0}}:= \nu_{0} \otimes \beta,  \,\, \,\,T^{X_{0}}:= B^{X_{0}}\rightarrow B^{X_{0}}, (x_{0}, b)\mapsto ( x_{0}b_{1},Tb)
\end{align*}
is measure-preserving and ergodic \cite[Proposition 2.14]{BQRW}. An application of Birkhoff  Theorem to the map $\widetilde{\sigma} : X_{0} \times B\rightarrow \R, (x_{0}, b)\mapsto \sigma(x_{0}, b_{1})$ yields  the announced convergence : for $\beta$-almost every $b\in B$, 
$$\frac{1}{n}\sigma(x_{0}, b_{1}\dots b_{n})= \frac{1}{n}\sum_{k=0}^{n-1}\widetilde{\sigma}\circ (T^{X_{0}})^k(x_{0}, b)\underset{n \to +\infty}{\longrightarrow} \beta^{X_{0}}(\widetilde{\sigma})=\nu_{0}\otimes\mu(\sigma)$$

\bigskip

Let us now deal with the  \underline{first case}. To lighten the notations, we will sometimes denote by $x$ a point of $X_{0}$ (and not $X$).  

\smallskip
The following lemma allows   additional assumptions on our drift cocycle $\sigma$.

\begin{lemme}\label{icontinu}
To prove \Cref{TH-deriv}, one can assume that $\sigma: X_{0}\times G \rightarrow \R^d$ is bounded on $X_{0}\times \supp \mu$ and continuous.
\end{lemme}

\smallskip

\begin{proof}
We noticed (\ref{fund}) that the drift of a $\mu$-trajectory does not depend on the choice of the pair $(D,i)$ made to quantify it. Hence, we just need to show that for a suitable pair $(D,i)$, the cocycle $\sigma$  is bounded on $X_{0}\times \supp \mu$ and continuous.

We choose $D$ such that its  boundary $\partial D$ is compact. This is possible because \Cref{TH-deriv} assumes that no cusp of $M_{0}$ is unfolded in $M$. 

We now construct $i$. Define a first draft by setting $j: M\rightarrow \Z^d$ the map  such that for all  $k\in \Z^d$,  $j_{| \overset{\circ}{D}+k}\equiv k$ and $j_{|M-\bigcup_{k\in \Z^d} \overset{\circ}{D}+k}\equiv0$. The map $j$ is $\Z^d$-equivariant on $\bigcup_{k\in \Z^d} \overset{\circ}{D}+k$, and is not continuous.  Let $m_{G} \in \mathcal{M}(G)$  be a left Haar measure on $G$ and $U\subseteq G$ a connected relatively compact open neighborhood of the neutral element of $G$ such that $U$ is left-K-invariant and $m_{G}(U)=1$. Seeing   the map $j$ as right-$K$-invariant map from $X$ to $\Z^d$, an noticing it is locally bounded, one can set 
$$i : X \rightarrow \R^d,\,\, x\mapsto \int_{U}j(xg)\,dm_{G}(g)$$
This map $i$ is  right-$K$-invariant, bounded on $D$, and  $\Z^d$-equivariant. The second assertion is true because if a translate $D+k$ has non-negligible intersection with $D.U$, then using the connectedness of $U$, one infers $D+k\cap (\partial D). U\neq 0$ implying that $k$ is bounded as $\partial D$ is compact. To check the $\Z^d$-equivariance of $i$, reason as follows. Let $x\in X$, $k\in \Z^d$. By definition, the map  $j$ is $\Z^d$-equivariant on the union $\bigcup_{k\in \Z^d} \overset{\circ}{D}+k$, which has full measure in $X$. Hence  for $m_{G}$-almost every $g\in G$, $j((x+k)g)=j(xg)+k$. Integrating on $U$, we obtain the equivariance of $i$. 

\bigskip
Let us see that the cocycle $\sigma$ associated to $(D,i)$ is continuous. It is enough to prove that $i$ is continuous.  Consider  points $x,y \in X$, and let $\widetilde{x}, \widetilde{y} \in G$ be  some lifts in $G$.  
\begin{align*}
||i(x)-i(y)||&= ||\int_{U}j(xg)-j(yg)\,dm_{G}(g)||\\
&=||\int_{U}j(xg)\,dm_{G}(g) -\int_{ \widetilde{x}^{-1}\widetilde{y} U}j(xg)\,dm_{G}(g)||\\
&\leq m_{G}(U \Delta \widetilde{x}^{-1}\widetilde{y} U)\sup_{xU\cup yU}||j||
\end{align*} 

If  $(y_{n})_{n\in \N}\in X^\N$ is a sequence that converges to $x$, then one can choose lifts  $(\widetilde{y}_{n})_{n\in \N}\in G^\N$ converging to $\widetilde{x}$ and in this case $m_{G}(U \Delta \widetilde{x}_{n}^{-1}\widetilde{y}_{n}U)\rightarrow 0$. Moreover, $\sup_{n\geq 0}\sup_{x_{n}U\cup y_{n}U}||j||$ is finite as $\bigcup_{n}x_{n}U\cup y_{n}U$ is relatively compact in $X$ and $j$ is locally bounded. Hence, the map $i$ is continuous.

Let us show  that  the cocycle $\sigma$ associated to $(D,i)$ is bounded on $\supp \mu \times X_{0}$. It is enough to sow that the index of position $i$ is bounded on the set $D. \supp \mu:=\{xg, x\in D, g\in \supp \mu\}$. To see this, notice that the compactness of  the support of $\mu$ and of the boundary of $D$ implies that the set $D.\supp \mu$ is covered by a finite number of translates $(D+k)_{k\in \Z^d}$. This leads to the result as $i$ is bounded on each of these translates.

\end{proof}

In view of \Cref{icontinu}, \emph{we assume from now on that the drift cocycle $\sigma$ is   bounded on $ X_{0}\times \supp \mu$ and continuous}. The following lemma reduces the study of the drift  to the study of  the $\mu$-averages of $\sigma$ on  $X_{0}$.

\begin{lemme}\label{moycoc}
For all $x \in X_{0}$, $\beta$-almost every $b\in B$, one has the convergence :

$$\frac{1}{n}\sigma(x, b_{1}\dots b_{n})-   \frac{1}{n}\sum_{k=0}^{n-1}\int_{G}\sigma(xb_{1}\dots b_{k},g) d\mu(g)\underset{n\to +\infty}{\longrightarrow}0$$ 

\end{lemme}

\bigskip

\begin{proof}
It is a corollary of a strong version of the law of large numbers. A proof is given in \cite[Proposition 3.2]{BQRW}. 
\end{proof}
\bigskip

We now fix a point  $x\in X_{0}$  whose $\Gamma_{\mu}$-orbit is infinite and we show that for $\beta$-almost every $b\in B$,
$$\frac{1}{n}\sigma(x, b_{1}\dots b_{n})\underset{n\to +\infty}{\longrightarrow}0$$ 

 \smallskip
  Setting $\varphi : X_{0}\rightarrow \R^d, x\mapsto \int_{G}\sigma(x,g)d\mu(g)$,   \Cref{moycoc} has reduced the problem to showing that for $\beta$-almost every  $b\in B$
 
 $$\frac{1}{n}\sum_{k=0}^{n-1}\varphi( xb_{1}\dots b_{k})\underset{n\to +\infty}{\longrightarrow}0$$

This sum can be controlled using the following result. It is due to Benoist-Quint and states that any infinite $\mu$-trajectory on $X_{0}$ must equidistribute toward  the Haar probability measure $\lambda_{0}$.

\begin{lemme}\label{convBQ}
For all $x\in X_{0}$ with infinite $\Gamma_{\mu}$-orbit, $\beta$-almost every  $b\in B$, one has the  weak-$\star$ convergence:
$$\frac{1}{n}\sum_{k=0}^{n-1}\delta_{xb_{1}\dots b_{k}}\underset{n\to +\infty}{\longrightarrow} \lambda_{0}$$
\end{lemme}

\begin{proof}
This result is obtained by combining Theorems $1.1$ and $1.3$ of \cite{BQIII} as well as corollary 1.2 of \cite{BQI}. The first theorem states that for all $x\in X_{0}$, the $\Gamma_{\mu}$-orbit  of $x$ has homogeneous closure, i.e. is of the form $\widebar{x\Gamma_{\mu}}=xH$ where $H$ is a closed subgroup of $G$, and also carries a (unique)  $H$-invariant probability measure $\nu_{x}$.  The second theorem adds that for all $x\in X_{0}$, $\beta$-almost every  $b\in B$, the sequence of probability measures $(\frac{1}{n}\sum_{k=0}^{n-1}\delta_{xb_{1}\dots b_{k}})_{n\geq 1}$  weak-$\star$ converges  to $\nu_{x}$. If the $\Gamma_{\mu}$-orbit of $x$ is infinite, then $\nu_{x}$ cannot have any atom, and corollary 1.2 of \cite{BQI} implies that $\nu_{x}$ has to be $G$-invariant, hence equal to $\lambda_{0}$. 
\end{proof}
\bigskip

Our additional assumptions on $\sigma$  imply that $\varphi$ is continuous and bounded. We infer from \Cref{convBQ} that  
$$\frac{1}{n}\sum_{k=0}^{n-1}\varphi( xb_{1}\dots b_{k})\underset{n\to +\infty}{\longrightarrow} \lambda_{0}(\varphi)$$ There remains to show the nullity of the drift from the point of view of  $\lambda_{0}$.

\begin{lemme}\label{derivleb}
for all $g \in G$, $$\int_{X_{0}}\sigma(x,g)\,d\lambda_{0}(x)=0$$
In particular, $\lambda_{0}(\varphi)=0$. 
\end{lemme}

\begin{proof}
Set $\phi : G \rightarrow \R^d, g \rightarrow \int_{X_{0}}\sigma(x,g)d\lambda_{0}(x)$. Let  $g,h\in G$. The cocycle relation and the $G$-invariance of $\lambda_{0}$ yield :
\begin{align*}
\phi(gh)&= \int_{X_{0}}\sigma(xh,g)+ \sigma(x,h)\,d\lambda_{0}(x)\\
&=\int_{X_{0}}\sigma(x,g)+ \sigma(x,h)\,d\lambda_{0}(x)\\
&=\phi(g)+\phi(h)
\end{align*}
Hence the map $\phi$ is a morphism of groups. As $G$ coincides with its commutator $[G,G]$, we must have $\phi(G)=0$. 
\end{proof}

 This concludes the proof of \Cref{TH-deriv}.
\bigskip

\section{Drift at almost every point} \label{sec3}

Let $M=\Lambda \backslash G/K$ be a $\Z^d$-cover of a finite volume - real or complex  - hyperbolic manifold $M_{0}=\Lambda_{0} \backslash G/K$ as in \Cref{sec1.1}, and $X=\Lambda\backslash G$, $X_{0}=\Lambda_{0}\backslash G$ their associated frame bundles. Our goal in  \Cref{sec3} is to describe the drift of a random walk on $X$ for almost every starting point. The context is more general than in \Cref{sec2}  as \emph{we authorize the cusps of $M_{0}$ to be unfolded in $M$}. This makes the drift cocycle harder to control : it might be unbounded, sometimes not even integrable. We will show that the case of non-integrability occurs only if $G= SO_{e}(1,2)$ and leads to radically different behaviors for the drift. More precisely, we will see in this case that if $(x_{n})$ denotes a random trajectory on $X$ and $i(x_{n})\in \R^d$ its successive positions, then the sequence $\frac{1}{n}i(x_{n})$ does not converge in $\R^d$ but accumulates over a subspace of $\R^d$ spanned by the directions of translations above the unfolded cusps.  

\bigskip

As above, we  use the following  \underline{notations}. The walk on $X$ is induced by a probability measure $\mu$  on $G$ whose support is compact and generates a Zariski-dense semigroup $\Gamma_{\mu}$ in $G$. We fix a good fundamental domain $D\subseteq M$  for the $\Z^d$-action on $M$,  a measurable map $i: M\rightarrow \R^d$  that is $\Z^d$-equivariant  and  bounded on $D$. We also see $i$ as a $\Z^d$-equivariant $K$-invariant  map on $X$ and write $\sigma : X_{0}\times G \rightarrow \R^d, (x+\Z^d, g)\mapsto i(xg)-i(x)$ the associated drift cocycle. We also set $B=G^{\N^\star}, \beta=\mu^{\otimes \N^\star}$.

\subsection{ Case 1 :  $G\neq SO_{e}(1,2)$}
\label{Case1}

In this section, we put aside the case of walks on hyperbolic surfaces and  show that in any other context the drift of a random walk on $X$ is almost everywhere null.  

\begin{th.}\label{drift-non path}
 If $G\neq SO_{e}(1,2)$, then for almost every $x_{0}\in X_{0}$, $\beta$-almost every $b\in B$, 
$$\frac{1}{n}\sigma(x_{0}, b_{1}\dots b_{n})\underset{n\to +\infty}{\longrightarrow}0 $$

\end{th.}

The main point is that the drift cocycle $\sigma : X_{0}\times G\rightarrow \R^d$ is integrable for the measure $\lambda_{0}\otimes \mu$, where $\lambda_{0}$ denotes the Haar probability measure on $X_{0}$. This result is already known in the case where $G=SO_{e}(1,m)$ \cite[remark 1]{Fra} but we have not found a reference for the case $G=SU(1,m)$ so we will give a proof.  Once the integrability is established, we can use  Birkhoff Ergodic Theorem and \Cref{derivleb} to conclude.  In order to show integrability, notice that the drift cocycle $\sigma$ is bounded on compact subsets of $X_{0}\times G$. Hence it all comes down to a good control of $\sigma$ for a starting point in  the cusps of $X_{0}$.

In the rest of \Cref{Case1}, we consider one of the cusps $\mathscr{C}_{j}$ of $M_{0}$, and  call it $\mathscr{C}$. We will rather see it as the corresponding  right-$K$-invariant subset of $X_{0}$. 
As in Section \ref{fund}, one can identify (up to negligible subsets) the cusp $\mathscr{C}$ with the product $F^0\times \R_{>0}\times K$.  We first use these coordinates to give an explicit  description of the Haar measure restricted to $\mathscr{C}$. It will be convenient to  assume that the element $v_{0}\in \mathfrak{a}$ parametrizing the Cartan subspace $\mathfrak{a}$ has norm $1$, i.e. $B(v_{0},v_{0})=1$, and that the translation parameter of a lift of $\mathscr{C}$ in $G$ is null, i.e. $g_{j}=e$ (see \ref{sec1.2} for notations).

\begin{lemme} \label{Haar}
Denote by $\pi : G \rightarrow X_{0}$ the canonical projection, and set $$\pi_{\mathscr{C}} : F^0\times \R_{>0}\times K \rightarrow \mathscr{C}, (n, t, k)\mapsto \pi(na_{t}k)$$
the coordinate map of $\mathscr{C}$. Then 
$$\pi_{\mathscr{C} \star}(dn \otimes e^{-\rho t}dt \otimes dk)=\lambda_{0|\mathscr{C}} $$
where $dn$, $dk$ are Haar measures on $N$, $K$, and $\rho=m-1$ if $G=SO_{e}(1,m)$ or $\rho=2m$ if $G=SU(1,m)$. 
\end{lemme}

\begin{proof}
We just need to express the Haar measure $\lambda$ of $G$ in terms of the Iwasawa decomposition $N\times \R \times K$. As $\lambda$ is left $N$-invariant and right $K$-invariant, $\lambda$ is of the form 
$$\lambda = dn \otimes \nu \otimes dk $$
where $\nu$ is some measure on $\R$. Denoting by $c_{g}$ the conjugation by $g$, one has 
$a_{t \star} (dn \otimes \nu) = c_{a_{t} \star} dn \otimes a_{t \star}\nu$. The image measure $c_{a_{t} \star}dn$ is still $N$-invariant so it must be a multiple of $dn$, let's say  $c_{a_{t} \star}dn=e^{\rho(t)}dn$. Necessarily, the coefficient $e^{\rho(t)}$ is the inverse of the determinant of $\text{Ad}a_{t}$ acting on $\mathfrak{n}$. If $G=SO_{e}(1,m)$, one gets $\rho(t)=-(m-1) t$ and if $G=SU(1,m)$, then $\rho(t)=-2m t$ (remember we normalized $v_{0}$).  As $dn \otimes \nu$ is left $a_{t}$-invariant, we conclude that $a_{t \star}\nu= e^{(m-1)t}\nu$ in the first case, and $a_{t \star}\nu= e^{2mt}\nu$ in the second. This leads to the announced result. 

\end{proof}

The next lemma controls the behavior of $\sigma$ along the cusps.

\begin{lemme} \label{estimation-sigma}
There exists constant $C>0$ such that  for all $n\in F_{0}$, $t>0$, $k\in K$, $g\in \supp \mu$,  one has $$\sigma(\pi(na_{t}k), g)\leq C e^{rt}$$
where $r=1$ if $G=SO_{e}(1,m)$ or $r=2$ if $G=SU(1,m)$ 
\end{lemme}

\begin{proof}
Endow $N\equiv NK/K$ with the riemannian metric induced by $G/K$ and denote by $d_{N}$ the corresponding left $N$-invariant distance.  For $R>0$, write 
$$B_{N}(F_{0}, R):=\{n\in N, d_{N}(n, F_{0}) <R\}$$
the set elements in $N$ whose distance to $F_{0}$ is strictly less than $R$. We are going to show there exists a constant $C_{1}>0$ such that for all $R\geq 1$, 
\begin{align*}
\sup_{x,y\in B_{N}(F_{0}, R)}|i(x)-i(y)| \leq C_{1} R \tag{$\ast$}
\end{align*}
where $i : M \rightarrow \R^d$ is identified here with its lift to $G$. Let us see first how we conclude the proof from here. By ($\ast$) and the property 3) of good fundamental domains (\ref{fund}), it is enough to show  there exists a constant $C_{2}>0$ such that for every $n\in F_{0}, t>0, k\in K, g\in \supp \mu$, one has $na_{t}kg = n_{1} a_{t_{1}}k_{1}$ with $n_{1}\in B_{N}(F_{0}, C_{2}e^{rt})$. To this end, write $kg= n_{2} a_{t_{2}}k_{2}$ the Iwasawa decomposition of $kg$.  Then 
$$n_{1}=n a_{t}n_{2} a_{-t}= n \exp(\text{Ad}(a_{t})(Y))$$
 where $Y\in \mathfrak{n}$ is the logarithm of $n_{2}$. The norm $||Y||$ is uniformly bounded by a constant $C_{2}$ when $kg$ varies in the compact set $K\times \supp \mu$. As $\text{Ad}(a_{t})$ acts symmetrically on $\mathfrak{n}$ with eigenvalues $e^t$ if $G=SO_{e}(1,n)$ and $e^{t}, e^{2t}$ if $G=SU(1,n)$, we infer that $||\text{Ad}(a_{t}) Y||\leq C_{2}e^{rt}$ and finally get $n_{1}\in B_{N}(F_{0}, C_{2}e^{rt})$.

\bigskip

We now prove the inequality ($\ast$). The neighborhood $B_{N}(F_{0}, 1)$ of $F_{0}$ is relatively compact, hence it is covered by a finite number of $N^\Lambda$-translates of $F_{0}$ :  $B_{N}(F_{0},1)\subseteq \bigcup_{i=1}^{k}g_{i}F_{0}$ with $g_{i}\in N^\Lambda$. As $d_{N}$ is left $N$-invariant, one has for any $g_{0}\in N^\Lambda$, the inclusion $B_{N}(F_{n_{0}},1)\subseteq \bigcup_{i=1}^{k}g_{0}g_{i}F_{0}$.

Set 
$$C_{1}\geq 2\sup_{x,y\in \bigcup_{i=1}^{k}g_{i}F_{0}}|i(x)-i(y)|$$

Notice that for $g\in N^\Lambda$, $x\in N$, one has $i(gx)=i(g)+i(x)$.  Hence, for any $g_{0} \in N^\Lambda$, one has the equivalent inequality 
$$C_{1} \geq 2\sup_{x,y\in \bigcup_{i=1}^k g_{0}g_{i}F_{0}}|i(x)-i(y)|$$
which implies that the variation of $i$ on any $d_{N}$-ball of radius $1$ is bounded by  $\frac{1}{2}C_{1}$. 

Now we can infer ($\ast$) as follows. Assume first that  $R\geq 1$ is an integer. Then if $x, y \in B(F_{0}, R)$,  one can find a sequence $x_{0}, \dots, x_{2R} \in N$ such that $x_{0}=x, x_{2R}=y$ and $d_{N}(x_{j},x_{j+1})<1$. We just showed that $|i(x_{j}) -i(x_{j+1})|\leq \frac{1}{2}C_{1}$. Hence $|i(x)-i(y)|\leq C_{1}R$. Replacing $C_{1}$ by $2C_{1}$, one get ($\ast$) for any $R\in [1,+\infty[$.
\end{proof}

\bigskip

The previous lemmas yield the announced  integrability of $\sigma$.
\begin{prop}\label{sigma-int}
 If $G\neq SO_{e}(1,2)$, then the drift cocycle $\sigma : X_{0}\times G \rightarrow \R^d$ is $\lambda_{0}\otimes \mu$-integrable
\end{prop}

\begin{proof}
The cocycle $\sigma$ is bounded on compact subsets of $X_{0}\times G$. Hence, we just need to show it is integrable on $\mathscr{C}\times G$ for every cusp $\mathscr{C}$ of $X_{0}$.  Using Lemmas \ref{Haar} and  \ref{estimation-sigma}, we can write :
\begin{align*}
\int_{\mathscr{C}\times G}|\sigma(x,g)|\,d\lambda_{0}(x) d\mu(g)&=\int_{F_{0}\times \R_{>0}\times K\times G}|\sigma\left( \pi(na_{t}k),g  \right)| \,dn  \,e^{-\rho t}dt \, dk\,d\mu(g)\\
&\leq \int_{F_{0}\times \R_{>0}\times K\times G} Ce^{(r-\rho)t} \,dn \, dt  \,dk \,d\mu(g)\\
&<\infty
\end{align*}
where the finiteness of the integral comes from the inequality $r<\rho $ that is true when $G\neq SO_{e}(1,2)$. 
\end{proof}

\bigskip

We can now conclude that the drift is almost everywhere null if $G\neq SO_{e}(1,2)$.

\begin{proof}[Proof of \Cref{drift-non path}]
The Haar probability measure $\lambda_{0}$ on $X_{0}$  is stationary and ergodic for the $\mu$-random on $X_{0}$. Hence the dynamical system $(B^{X_{0}}, \beta^{X_{0}}, T^{X_{0}})$ defined by 
\begin{align*}
B^{X_{0}}:= X_{0} \times B, \,\, \,\,\beta^{X_{0}}:= \lambda_{0} \otimes \beta,  \,\, \,\,T^{X_{0}}:= B^{X_{0}}\rightarrow B^{X_{0}}, (x_{0}, b)\mapsto ( x_{0}b_{1},Tb)
\end{align*}
is measure-preserving and ergodic \cite[Proposition 2.14]{BQRW}. Moreover, according to \Cref{sigma-int}, the map $\widetilde{\sigma} : X_{0} \times B\rightarrow \R, (x_{0}, b)\mapsto \sigma(x_{0}, b_{1})$ is $\beta^X_{0}$-integrable. Hence  Birkhoff  Theorem implies  that for $\lambda_{0}\otimes \beta$-almost every $(x_{0},b)\in X_{0}\times B$,  

$$\frac{1}{n}\sigma(x_{0}, b_{1}\dots b_{n})\underset{n\to +\infty}{\longrightarrow} \int_{X_{0}\times G}\sigma(x,g) d\lambda_{0}(x)d\mu(g)$$ 
By \Cref{derivleb}, this integral must be zero.

\end{proof}

\subsection{ Case 2 :  $G= SO_{e}(1,2)$}
\label{Case2}
The previous section has reduced our study to the case where $G=SO_{e}(1,2)$. In this context, the space $M=\Lambda\backslash G/K$ is a $\Z^d$-cover of the  finite volume  hyperbolic surface $M_{0}=\Lambda_{0}\backslash G/K$ and $X, X_{0}$ correspond to their unit tangent bundles. We will rather use the notations $S=M$, $S_{0}=M_{0}$, $T^1S=X$, $T^1S_{0}=X_{0}$. 

\bigskip
The crucial difference with  \Cref{Case1} is that the drift cocycle $\sigma : T^1S_{0}\times G\rightarrow \R^d$ is not $\lambda_{0}\otimes \mu$-integrable whenever a cusp of $S_{0}$ is unfolded in $S$. Let us give a brief justification.  Denote by $\mathscr{C}$ such a cusp. We just need to check that for $g\in G-K$,  the drift map $\sigma(.,g)$ is not integrable along $\mathscr{C}$.  Let $\mathbb{H}$ be the hyperbolic half-space. The preimage of $\mathscr{C}$ in $S$ can be identified with some upper level set $\mathbb{H}_{t}:=\{z\in \mathbb{H},\, \text{Im} z> e^t\}$ where $t\in \R$.  For  most vectors $x=(z,v)\in T^1\mathbb{H}_{t}$, the translate $xg$ has its abscissa of the same  order as $\text{Im}z$, which leads to $||\sigma(x,g)||\simeq \text{Im}z$ as we consider good fundamental domains (\ref{fund}).  But the riemannian metric on $\mathbb{H}_{t}$ is precisely $\frac{1}{(\text{Im}z)^2}\langle.,. \rangle_{eucl}$, so $$\int_{T^1\mathbb{H}_{t}} ||\sigma(x,g)|| dx \simeq \int_{]e^t,+\infty[} y\frac{1}{y^2} dy=+\infty$$

\bigskip
The non-integrability of $\sigma$  hints at  radically different conclusions : the drift of a typical $\mu$-trajectory is no longer null, but accumulates over a whole  sub-vector space $E_{C}$ in $\R^d$, generated by the directions of translation above the unfolded cusps in $S_{0}$. More precisely, denote by $\mathscr{C}_{1}, \dots, \mathscr{C}_{q}$ the cusps of $S_{0}$,  write $v_{1}, \dots,v_{q}\in \Z^d$ the translations obtained by lifting to $S$  simple closed curves homotopic to the cusps, and set $E_{C}=\text{Vect}_{\R}\{v_{1},\dots v_{q}\}$. We prove the following. 

\smallskip
\begin{th.}[Drift]\label{drift-hypsurf}
The drift of the $\mu$-walk on $T^1S$ is almost everywhere equal to  $E_{C}$ :  for almost every $x_{0}\in T^1S_{0}$, for $\beta$-almost every $b\in B$,
$$\left\{\text{accumulation points of the sequence $\left(\frac{1}{n}\sigma(x_{0}, b_{1}\dots b_{n})\right) _{n\geq 1}$}\right\} = E_{C} $$
\end{th.}

\bigskip
Our proof of \Cref{drift-hypsurf} will also yield a characterization of $\Z^d$-covers of finite volume hyperbolic surfaces for which the $\mu$-walk is recurrent.

\begin{th.}[Recurrence]\label{rec-hypsurf}

The $\mu$-walk on $T^1S$ is recurrent and ergodic if and only if $d=1$, or $d=2$ and $\dim E_{C}=0$; and it is transient otherwise.
\end{th.}

\smallskip

In this statement, everything is to be understood from the point of view of a Haar measure $\lambda$ on $T^1S$. The \emph{recurrence} (resp. \emph{transience}) means that for $\lambda$-almost every starting point $x\in T^1S$, $\beta$-almost every $b\in B$, the trajectory $(xb_{1}\dots b_{n})_{n\geq 0}$ subconverges to $x$ (resp. leaves every compact set of $T^1S$). If the walk is recurrent, we specify it is \emph{ergodic} if for any measurable subset $E\subseteq T^1S$ with  $\lambda(E)>0$ and $\lambda\otimes \beta$-almost every $(x,b)\in T^1S\times B$, the trajectory $(xb_{1}\dots b_{n})_{n\geq 0}$ meets $E$ infinitely many times. It is equivalent to say that the action of $\Gamma_{\mu}$ on $T^1S$ is ergodic for $\lambda$ (see  \cite{Tim-these}, Sections 1.3 and 3.1.1).

\Cref{rec-hypsurf} extends a result of  Conze and Guivarc'h \cite[Proposition 4.5]{ConGui} that deals with  the case where the measure $\mu$ is symmetric (i.e. invariant under $g\mapsto g^{-1}$) and the base $S_{0}$ is compact. We recently generalized \Cref{rec-hypsurf}   by proving that  if $X=\Lambda\backslash G$ is a rank-one homogeneous space (not necessarily an abelian cover with finite volume base) and $\mu$ a  Zariski-dense probability measure on $G$ with finite first moment then the $\mu$-walk and the geodesic flow on $X$ are either both recurrent ergodic, or both transient \cite{Tim-recurrence}. The proof is however much more involved than in the present case, and uses a different technique, approximating  $\mu$-trajectories by their asymptotic geodesic ray.

\bigskip

\subsubsection{A central limit theorem}

We circumvent the non-integrability of the drift cocycle $\sigma$ by estimating the limit law of the variable  $\frac{1}{n}\sigma(x,b_{1}\dots b_{n})$ as $n$ goes to infinity. As above, we set  $B=G^{\N^\star}$, $\beta=\mu^{\N^\star}$ and denote by $\lambda_{0}$  the Haar probability measure on $T^1S_{0}$.

\begin{prop}[CLT]\label{Cauchy} Let $\nu_{n}$ be the image probability measure of $\lambda_{0} \otimes \beta$ under the map $ T^1S_{0} \times B \rightarrow \R^d,\, (x,b)\mapsto  \frac{1}{n}\sigma(x,b_{1}\dots b_{n})$. Then the sequence $(\nu_{n})_{n\geq0}$ weak-$\star$ converges to a centered Cauchy distribution on   $E_{C}$.   
\end{prop}

Recall that a centered Cauchy distribution on $\R$ is a probability measure of the form  
$$m_{c}=\frac{c}{\pi (c^2+t^2)}dt $$
where $c>0$  is a positive parameter. More generally, given a finite dimensional real vector space $E$, a probability measure $m$ on $E$ is a centered Cauchy distribution if there exists a linear isomorphism $\psi : E\overset{\sim}{\rightarrow} \R^k$ for which the image  $\psi_{\star}m$ of $m$ is of the form $m_{c_{1}}\otimes \dots \otimes m_{c_{k}}$ where the $m_{c_{i}}$ are centered Cauchy distributions on $\R$.


\bigskip
The proof of  \Cref{Cauchy} relies on a result of Enriquez and Le Jan, which estimates how the index of position $i$ varies in law under the action of the geodesic flow. 

\begin{lemme}\label{Cauchy1}
For $t>0$, denote by $\nu'_{t}$ the image probability measure of $\lambda_{0}$ by the  map $T^1S_{0}\rightarrow \R^d, x \mapsto \frac{1}{t}\sigma(x,a_{t})$. Then  $(\nu'_{t})_{t>0}$ weak-$\star$ converges to a  centered Cauchy distribution $m_{C}$ on $E_{C}$  as $t$ goes to $+\infty$. 
\end{lemme}

\begin{proof}
This result comes from the article \cite{Enr-LeJan}  of Enriquez and Le Jan, who consider the geodesic flow on a hyperbolic surface of finite volume and prove that the normalized winding numbers of the flow around the different cusps behave asymptotically as a product of independent centered Cauchy distributions. Let us give more details. We first decompose the cocycle $\sigma$ to count separately the moves occurring above the unfolded cusps, and above their complementary subset.  Remember the partition $S_{0}=L_{0}\sqcup  \bigsqcup_{j=1}^q \mathscr{C}_{j}$ of \ref{sec1.2} and suppose that $(\mathscr{C}_{j})_{j=1, \dots, k}$ are the  unfolded cusps of $S_{0}$. Write $p :S \rightarrow S_{0}$ the covering map, $H=p^{-1}(\bigsqcup _{j\leq k} \mathscr{C}_{j})$ and $F=p^{-1}(L_{0}\cup \bigsqcup_{j>k} \mathscr{C}_{j})$. We then have a decomposition $S=H \sqcup  F$, with $H$ open. 
For $t>0$, $x_{0}\in T^1S_{0}$, $x\in (Tp)^{-1}(x_{0})$,  set $]s_{0},s_{1}[\sqcup  \dots \sqcup  ]s_{2n},s_{2n+1}[= \{s \in ]0,t[, \,xa_{s}\in T^1H\}$, and define  
$$\sigma_{H}(x_{0},t) := \sum^n_{k=0}i(xa_{s_{2k+1}})-i(xa_{s_{2k}}) $$ 
$$ \sigma_{F}(x_{0},t):= \sigma(x_{0},a_{t})-\sigma_{H}(x_{0},t) $$

\bigskip

Hence $\sigma(x_{0},a_{t})=\sigma_{H}(x_{0},t)+\sigma_{F}(x_{0},t)$ where $ \sigma_{H}, \sigma_{F}$ are cocycles on  $X_{0}\times \R$.  The result stated in \cite{Enr-LeJan} implies that the family of probability measures $\frac{1}{t}\sigma_{H}(.,t)_{\star} \lambda_{0}$ weak-$\star$ converges to a centered Cauchy distribution $m_{C}$ on $E_{C}$ as $t$ goes to infinity. We need to show this is also true for the cocycle $\sigma$.

We first check that  for almost every $x\in T^1S_{0}$ , the term $\sigma_{F}(x,t)$ converges to zero as $t$ goes to infinity. Notice that the cocycle $\sigma_{F}$ is bounded on  $T^1S_{0}\times [0,1]$, hence the  ergodicity of the geodesic flow on $(T^1S_{0}, \lambda_{0})$ yields that for  $\lambda_{0}$-almost every $x\in T^1S_{0}$, 
$$\frac{1}{t}\sigma_{F}(x,t) \underset{t\to +\infty}{\longrightarrow} \int_{T^1S_{0}} \sigma_{F}(x,1)\,d\lambda_{0}(x)$$
We now check the limit is zero.  Writing $R_{\pi}$ the  rotation of angle  $\pi$ with fixed base point, we have  by $G$-invariance of $\lambda_{0}$ that 
 $$\int_{T^1S_{0}}\sigma_{F}(x,1)\, d\lambda_{0}(x) = \frac{1}{2}\int_{T^1S_{0}}\sigma_{F}(x,1) +\sigma_{F}(xa_{1}R_{\pi},1) \,d\lambda_{0}(x)  $$
 The paths $(x a_{s})_{s\in [0,1]}$ and $(xa_{1}R_{\pi}a_{s})_{s\in [0,1]}$  travel along the same geodesic arc but in opposite directions, so  they induce opposite variations of index :
 $$\sigma_{F}(x,1)=-\sigma_{F}(xa_{t}R_{\pi},1) $$
 In particular the previous integral is null, and for   $\lambda_{0}$-almost every $x\in T^1S_{0}$, 
$$\frac{1}{t}\sigma_{F}(x,t) \underset{t\to +\infty}{\longrightarrow} 0$$

We now conclude that  $\frac{1}{t}\sigma(.,t)_{\star} \lambda_{0}$ weak-$\star$ converges the distribution $m_{C}$  as $t$ goes to infinity. If $\varphi \in C^0_{c}(\R^d)$ is a continuous function on $\R^d$ with compact support, the previous paragraph,  the uniform continuity of $\varphi$, and the dominated convergence theorem imply that :

\begin{align*}
\int_{T^1S_{0}}\varphi\left(\frac{1}{t}\sigma(x,a_{t})\right)\,d\lambda_{0}(x) &=\int_{T^1S_{0}}\varphi\left(\frac{1}{t}\sigma_{H}(x_{0},t)+\frac{1}{t}\sigma_{F}(x_{0},t)\right)\,d\lambda_{0}(x) \\
&=\int_{T^1S_{0}}\varphi\left(\frac{1}{t}\sigma_{H}(x_{0},t)\right)\,d\lambda_{0}(x) +o(1)\\
&\underset{t\to+\infty }{\longrightarrow} m_{C}(\varphi)
\end{align*}
which finishes the proof of \Cref{Cauchy1}.

\end{proof}

\begin{proof}[Proof of  \Cref{Cauchy}]
The proof relies on  \Cref{Cauchy1} and on the Cartan decomposition of elements in $G$, which provides a strong relation between the  $\mu$-walk and the geodesic flow on $T^1S$.  This decomposition states that every  $g\in G$ can be written as 
$$g=k_{g}a_{t_{g}}l_{g} $$
where $k_{g},l_{g}\in K$, $t_{g}\geq 0$. Moreover the number $t_{g}$ is unique and called the Cartan projection of $g$. 

Fix such a decomposition for each element of $G$ (in a measurable way). For $b\in B$, denote by $b_{1}\dots b_{n}=k_{n}(b)a_{t_{n}(b)}l_{n}(b)$ the decomposition of the product $b_{1}\dots b_{n}$. 

\begin{lemme} Let $\varphi\in C^0_{c}(\R^d)$. Then
$$\nu_{n}(\varphi)=\int_{T^1S_{0} \times B}\varphi\left(\frac{1}{n}\sigma(x,a_{t_{n}(b)} )\right) d\beta(b)d\lambda_{0}(x)$$ 
\end{lemme}

\begin{proof}
This comes from the $K$-invariance of the index of position  $i$ and of the probability measure $\lambda_{0}$.

\end{proof}

Let $\varphi\in C^0_{c}(\R^d)$. Denote by $\lambda_{\mu}>0$ the first Lyapunov exponent of the probability measure $\mu$ on $G$ \cite[Section 4.6]{BQRW}. The previous lemma and Fubini Theorem imply that
\begin{align*}
\nu_{n}(\varphi)&=\int_{B}\nu'_{t_{n}(b)}\left(\varphi(\frac{t_{n}(b)}{n}.) \right) \,d\beta(b) \\
&=\int_{B}\nu'_{t_{n}(b)} \left(\varphi (\lambda_{\mu}.) \right)\,d\beta(b) + \int_{B}\nu'_{t_{n}(b)}\left(\varphi(\frac{t_{n}(b)}{n}.)-\varphi(\lambda_{\mu}.) \right) \,d\beta(b)\\
\end{align*}

The positivity of the first Lyapunov exponent \cite[Corollary 4.32]{BQRW} imply that for $\beta$-almost every $b\in B$,  one has $t_{n}(b)\rightarrow +\infty$ as $n$ goes to infinity. Hence, the dominated convergence theorem and \Cref{Cauchy1} imply that the first term converges
 \begin{align*}
\int_{B}\nu'_{t_{n}(b)}\left(\varphi (\lambda_{\mu}.)\right) \,d\beta(b)\underset{n\to +\infty}{\longrightarrow} (\lambda_{\mu}m_{C})(\varphi) \tag{1}
\end{align*}

Let us check that the second term goes to $0$. Let $\varepsilon>0$. As $\varphi$ has compact support, there exists $\delta>0$ such that for all $\alpha\in [\lambda_{\mu}-\delta, \lambda_{\mu}+\delta]$, one has $||\varphi(\alpha.)-\varphi (\lambda_{\mu}.)||_{\infty}<\varepsilon/2$. By the law of large numbers for the norm  \cite[Lemma 4.27]{BQRW}, there exists an integer $n_{0}\geq 1$, such that if $n\geq n_{0}$, then 
$$\beta\{b\in B, \,\,\frac{t_{n}(b)}{n}\in [\lambda_{\mu}-\delta, \lambda_{\mu}+\delta]\}\geq 1-\frac{\varepsilon}{4||\varphi||_{\infty}} $$

For $n\geq n_{0}$, we deduce that   
\begin{align*}
| \int_{B}\nu'_{t_{n}(b)}\left(\varphi(\frac{t_{n}(b)}{n}.)-\varphi(\lambda_{\mu}. ) \right) \,d\beta(b)| &\leq   \int_{B} ||\varphi(\frac{t_{n}(b)}{n}.)-\varphi(\lambda_{\mu}.)||_{\infty}\,d\beta(b) \\
&\leq \varepsilon/2 + \varepsilon/2\\
&=\varepsilon \tag{2} \\
\end{align*}

Combining (1) and (2), we obtain : 
$$\nu_{n}(\varphi)\underset{n\to +\infty}{\longrightarrow} \lambda_{\mu}m_{C}(\varphi)  $$
which concludes the proof of  \Cref{Cauchy}.

\end{proof}

\bigskip

\subsubsection{Application to the drift and the recurrence}
\bigskip

We now use \Cref{Cauchy} to obtain  Theorems \ref{drift-hypsurf} and \ref{rec-hypsurf}. The two results are proven independently. 

\bigskip

{\bf Drift}
\smallskip

Let us show Theorem \ref{drift-hypsurf}. We first check that \Cref{Cauchy} implies that the set of accumulation points of a  typical sequence $\left(\frac{1}{n}\sigma(x_{0}, b_{1}\dots b_{n})\right)_{n\geq 1} $ contains the vector space $E_{C}$ spanned by the directions of translation above the unfolded cusps. 

\begin{lemme} \label{supEC}

For almost every $x_{0}\in T^1S_{0}$, for $\beta$-almost every $b\in B$, one has 
$$\left\{\text{accumulation points of the sequence $\left(\frac{1}{n}\sigma(x_{0}, b_{1}\dots b_{n})\right)_{n\geq 1}$}\right\} \supseteq E_{C} $$

\end{lemme}

\begin{proof}
We only need to prove that for a fixed $t\in E_{C}$, for almost every $x_{0}\in T^1S_{0}$, and $\beta$-almost every $b\in B$, the sequence $\left(\frac{1}{n}\sigma(x_{0}, b_{1}\dots b_{n})\right)_{n\geq 1}$ has a subsequence converging to $t$. Let $\varepsilon >0$. For $n\geq 1$, set
$$E_{n}=\{(x,b)\in T^1S_{0}\times B, \,\,||\frac{1}{n}\sigma(x,b_{1}\dots b_{n})-t|| \leq \varepsilon\}$$  
 \Cref{Cauchy} implies that the sequence $\left(\lambda_{0}\otimes \beta(E_{n})\right)_{n\geq 1}$ converges toward a positive number $\delta=\delta(t,\varepsilon)>0$. Hence  $\lambda_{0}\otimes \beta(\limsup E_{n})\geq \delta$. 
 
 Set
$$E=\{(x,b)\in  T^1S_{0}\times B, \,\,\liminf_{n\to +\infty} ||\frac{1}{n}\sigma(x, b_{1}\dots b_{n})-t|| \leq \varepsilon\}$$
 As $E\supseteq \limsup E_{n}$, one has  $\lambda_{0}\otimes \beta(E)\geq \delta$. However, the set $E$ is invariant under the transformation $  T^1S_{0}\times B\rightarrow T^1S_{0}\times B, (x,b)\mapsto (xb_{1}, Tb)$ which is measure-preserving and ergodic \cite[Proposition 2.14]{BQRW}. Hence $\lambda_{0}\otimes \beta(E)=1$. Choosing smaller and smaller $\varepsilon$, we obtain that  $\lambda_{0}\otimes \beta$-almost surely, 
$$\liminf_{n\to +\infty} ||\sigma(x, b_{1}\dots b_{n}) -t||=0$$
which concludes the proof. 
\end{proof}
\bigskip

We  now check  the reverse inclusion in  \Cref{drift-hypsurf}.  

\begin{lemme} \label{subEC}
For almost every $x_{0}\in T^1S_{0}$, for $\beta$-almost every $b\in B$, one has 
$$\left\{\text{accumulation points of the sequence $\left(\frac{1}{n}\sigma(x_{0}, b_{1}\dots b_{n})\right)_{n\geq 1}$}\right\} \subseteq E_{C} $$
 \end{lemme}

\begin{proof}

Fix a  subspace $E'\subseteq \R^d$   such that $$\R^d =E'\oplus E_{C} $$
and decompose $\sigma$ into $\sigma=\sigma'+\sigma_{C}$ where $\sigma'$ and $\sigma_{C}$ are cocycles respectively taking values in $E'$ and $E_{C}$. It is sufficient to prove that for almost every $x_{0}\in T^1S_{0}$, for $\beta$-almost every $b\in B$, 
\begin{align*}
\frac{1}{n}\sigma'(x,b_{1}\dots b_{n})\underset{n\to +\infty}{\longrightarrow} 0  \tag{3}
\end{align*}
 Let us first check that the cocycle $\sigma'$ is bounded on the set $T^1S_{0} \times \supp \mu $. As the family of translates $(D+k)_{k\in \Z^d}$ is locally finite, we have for every compact set $L\subseteq T^1S_{0}$,  
$$\sup \{||\sigma(x,g)||,\,\,x\in L,\,g\in \supp \mu\}<\infty$$
and this must remain true for the projection  $\sigma'$. 

We now need to show that $\sigma'(x,g)$ is uniformly bounded as $g$ varies in the support of $\mu$, and $x$ in a cusp of $S_{0}$.  Let $g=k_{g}a_{t_{g}}l_{g}$ be a Cartan decomposition of $g$ and $\gamma_{x,g} : [0, t_{g}]\rightarrow T^1S_{0}, t\mapsto xk_{g}a_{t}l_{g}$. One can see $\sigma(x,g)$ as the  variation of $i$ along a lift $\widetilde{\gamma}_{x,g} : [0, t_{g}]\rightarrow T^1S$ of the path $\gamma_{x,g}$. Choose $x$ high enough in the cusp so that  $\gamma_{x,g}$ in entirely included in the cusp for every $g\in \supp \mu$. As the domain $D$ is assumed to be good (see \ref{fund}), we can complete the path $\gamma_{x,g}$ into a path  $c_{xg}\gamma_{x,g}c_{x}$ contained in the cusp,  with extremal points in  the boundary of the cusp, and in a way that $c_{xg}$, $c_{x}$ lift to continuous paths in  $D$. Finally, complete $c_{xg}\gamma_{x,g}c_{x}$ into a closed path $\delta c_{xg}\gamma_{x,g}c_{x}$ by adding an arc $\delta$ taking values in the  boundary of the cusp, and winding less than once. As the  boundary of the cusp is compact, the variation of $i$ along a lift $\widetilde{\delta}$ in $T^1S$ is bounded above, by a constant $M>0$ that only depends on the cusp at study and the tiling $(D+k)_{k\in \Z^d}$. In conclusion the variation of $i$ along lifts to $T^1S$ of the paths ${\gamma}_{x,g}$ and $\delta c_{xg}\gamma_{x,g}c_{x}$ differ at most by $4 \sup_{D}|i|+ M$. As the second one is in $E_{C}$, we infer that the distance between $\sigma(x,g)$ and $E_{C}$, hence the norm of $\sigma'(x,g)$,  is uniformly bounded as $x$ varies in a cusp of $S_{0}$ and $g$  in the support of $\mu$.

The cocycle $\sigma'$ being bounded on $T^1S_{0}\times \supp \mu $, the ergodicity of the $\mu$-walk on $T^1S_{0}$ for the Haar measure $\lambda_{0}$ reduces the convergence (3)   to showing that $$ \int_{T^1S_{0} \times G}\sigma'(x,g)\,d\lambda_{0}(x)d\mu(g)=0$$
As in the proof of \Cref{derivleb}, this comes from the facts that  $G$ is simple and the map  $G\rightarrow \R^d, g\mapsto \int_{ T^1S_{0}}\sigma'(g,x)\lambda_{0}(x)$ is a group morphism.

\end{proof}

\begin{proof}[Proof of \Cref{drift-hypsurf}]
Combine Lemmas \ref{supEC} and \ref{subEC}. 
\end{proof}

\bigskip

{\bf Recurrence and ergodicity}
\smallskip

Schmidt-Conze Theorem \cite{Schmidt-rec} relates the recurrence of a skew-product to the asymptotic behavior in law of its iterates. The following particular case is  noteworthy. 

\begin{th*}[Schmidt-Conze]
Let $(Z,\mathcal{Z}, \mathbb{P})$ be a probability space, $R:Z\rightarrow Z$  a measure preserving ergodic automorphism. Let $d\in \{1,2\}$ and $f : Z\rightarrow \R^d$ a measurable map such that the sequence $(n^{-1/d}\sum^{n-1}_{k=0} f\circ R^k)_{n\geq 1}$ converges in law toward a  centered Cauchy distribution if $d=1$, or a centered Gaussian distribution if $d=2$.  

Then for $\mathbb{P}$-almost every $z\in Z$,  $$\liminf_{n\to+\infty} ||\sum^{n-1}_{k=0} f\circ R^k(z) ||=0$$ 

\end{th*}

\bigskip

We use \Cref{Cauchy} and Schmidt-Conze Theorem to obtain the first half of \Cref{rec-hypsurf}.

\begin{lemme}\label{rec}
If  $d=1$, or $d=2$ and $\dim E_{C}=0$, then the $\mu$-walk on $T^1S$ is recurrent and ergodic.

\end{lemme}

\begin{proof}

We first use Schmidt-Conze Theorem to check that   for almost every $x\in T^1S_{0}$, $\beta$-almost every $b\in B$, the sequence $\sigma(x, b_{1}\dots b_{n})$ sub-converges to zero :
\begin{align*}
\liminf_{n\to +\infty} ||\sigma(x, b_{1}\dots b_{n})||=0 \tag{4}
\end{align*}

Set $Z=T^1S_{0}\times G^\Z$, $\mathcal{Z}$ the product $\sigma$-algebra, $\mathbb{P}=  \lambda_{0} \otimes \mu^{\otimes \Z}$, and consider the automorphism $R:Z\rightarrow Z, (x,b)\mapsto (xb_{1},Tb)$ where $T$ stands for the two-sided shift $G^\Z$, i.e. $T : (b_{i})_{\in \Z} \mapsto (b_{i+1})_{i\in \Z}$.  The map $T$ is a measure-preserving ergodic automorphism of $(Z, \mathcal{Z}, \mathbb{P})$. Define also $f :Z\rightarrow \R^d,\, (x,b)\mapsto \sigma(x, b_{1})$ and notice that $$\sigma(x,b_{1}\dots b_{n})=\sum^{n-1}_{k=0} f\circ R^k(x,b)$$

If $d\in \{1,2\}$,  and $\dim E_{C}=0$, then the law of the variable  $(T^1S_{0}, \lambda_{0}) \rightarrow \R^d, x\mapsto  \frac{1}{\sqrt{t}}\sigma(x,a_{t})$ is known to converge to a centered Gaussian distribution on $\R^d$ as $t$ goes to infinity  \cite[pages 3,5,6]{SarLed}. As in \Cref{Cauchy}, we infer that the  law of the normalized drift cocycle  $(T^1S_{0}\times B, \lambda_{0}\otimes \beta) \rightarrow \R^d, (x,b)\mapsto \frac{1}{\sqrt{n}}\sigma(x,b_{1}\dots b_{n})$ also converges to a centered Gaussian distribution on $\R^d$ when $n$ goes to $+\infty$.  Schmidt-Conze Theorem  (possibly in the case of a degenerate Gaussian law on $\R^2$) yields (4).

If $d=\dim E_{C}=1$, the convergence (4) also holds, by the same argument,  combining \Cref{Cauchy} and Schmidt-Conze Theorem. 

\bigskip
 
We now infer the $\mu$-walk  on $T^1S$ is almost everywhere recurrent. Let $U\subseteq \R^d$ be an open ball  and  $\Omega:=i^{-1}(U)\subseteq T^1S$. According to the previous paragraph, almost every trajectory of the $\mu$-walk starting in $\Omega$ comes back infinitely many times to $\Omega$. This allows to define the first return random walk on $\Omega$, and the Haar measure restricted to $\Omega$ is stationary for this walk. As $\Omega$ has finite volume,  Poincaré Recurrence Theorem implies that the first return walk is recurrent on $\Omega$, which yields the recurrence of the $\mu$-walk on $T^1S$ for almost every starting point in $\Omega$. As the open ball $U$ is arbitrary,  we conclude that the $\mu$-walk is recurrent on $T^1S$.

Finally, we prove the ergodicity of the $\mu$-walk for the Haar measure $\lambda$ on $T^1S$. First notice that for any $c\in \R\setminus \{0\}$, the action of the discretized geodesic flow $a_{c}$ is  recurrent for almost every starting point on  $T^1S$ (this is proven in \cite{Rees81} when $S_{0}$ is compact, for the general case, argue  as above, applying Schmidt-Conze Theorem and \cite{SarLed} to  the geodesic flow instead of  random walks). By Hopf dichotomy, the action of $a_{c}$ must be $\lambda$-ergodic as well (see \cite[Theorem 7.4.3]{Aar} or \cite[lemme 4.4.2]{Tim-these}). As the semigroup $\Gamma_{\mu}$ generated by  the support of $\mu$ is Zariski-dense in $G$, it contains some loxodromic element $g_{0}$ (see \cite[Proposition 6.11]{BQRW}). By definition, this element can be written as a conjugate $g_{0} =g a_{c}g^{-1}$ where $g\in G$, $c\neq 0$.  Hence the action of $g_{0}$ on $T^1S$, and a priori the action of $\Gamma_{\mu}$, is also $\lambda$-ergodic. This proves the ergodicity of the $\mu$-walk on $T^1S$. 
  \end{proof}

\bigskip

\begin{lemme}\label{trans}
If  $(d, \dim E_{C})\notin \{(1,0), (1,1), (2,0)\}$ then the $\mu$-walk on $T^1S$ is transient.

\end{lemme}

\begin{proof}
The key is the local limit theorem for the geodesic flow  on $T^1S$  \cite{PanOh, Panlocalmixing}, according to which  there exists a constant $c_{0}>0$ such that for any open subsets $U,V\subseteq T^1S$, 
$$\lambda(Ua_{t}\cap V)\underset{t\to +\infty}{\sim}c_{0}\lambda(U)\lambda(V) \,t^{-\frac{1}{2}(d +\dim E_{C})}$$

We fix open subsets $U, V$ that are both right $K$-invariant and bounded, and we show that  for $\lambda$-almost every starting point $x\in U$, $\beta$-almost every $b\in B$, the sequence $(x b_{1}\dots b_{n})_{n\geq 0}$ meets $V$ only finitely many times. As in the proof of \Cref{Cauchy},  we choose for every $b\in B$, $n\geq 0$,  a Cartan decomposition $b_{1}\dots b_{n}=k_{n}(b)a_{t_{n}(b)}l_{n}(b)$  of the product $b_{1}\dots b_{n}$. By $K$-invariance of $U$, $V$ and $G$-invariance of the Haar measure $\lambda$, we have for $b\in B$, 

\begin{align*}
\int_{U}\sum_{n\geq0}1_{V}(x b_{1}\dots b_{n}) d\lambda(x) &= \sum_{n\geq 0}\int_{U}1_{V}(x a_{t_{n}(b)}) d\lambda(x)\\
&= \sum_{n\geq 0} \lambda(Ua_{t_{n}(b)}\cap V) 
\end{align*}

By the law of large numbers, we have for $\beta$-almost every $b\in B$ the asymptotic equivalence 
$t_{n}(b) \underset{n\to +\infty}{\sim}  n\lambda_{\mu}$ where $\lambda_{\mu}>0$ is the Lyapunov exponent of the walk (see \cite{BQRW} Lemma 4.27 and Corollary 4.32).
Hence, $$ \lambda(Ua_{t_{n}(b)}\cap V) \underset{n\to +\infty}{\sim}  c_{0}\lambda(U)\lambda(V) \,(n\lambda_{\mu})^{-\frac{1}{2}(d +\dim E_{C})}$$

As the assumptions of  \Cref{trans} mean that $d+\dim E_{C}\geq 3$, we infer that 
$$\int_{U}\sum_{n\geq0}1_{V}(x b_{1}\dots b_{n}) d\lambda(x) <\infty $$
In particular, for $\lambda$-almost every starting point $x\in U$, $\beta$-almost every $b\in B$, the sequence $(x b_{1}\dots b_{n})_{n\geq 0}$ meets $V$ only finitely many times, hence   the $\mu$-walk on  $T^1S$ is transient.  
\end{proof}

\begin{proof}[Proof of \Cref{rec-hypsurf}]
Combine Lemmas \ref{rec} and \ref{trans}. 
\end{proof}

\bibliographystyle{abbrv}

\bibliography{bibliographie}

\end{document}